\pdfoutput=1
\documentclass[11pt]{amsart}
\usepackage{geometry}                
\usepackage[parfill]{parskip}   
\usepackage{graphicx,color,amsthm}
\usepackage{amssymb}
\usepackage{epstopdf}
\usepackage{pdfsync}
\usepackage{caption}
\usepackage{subcaption}
\usepackage{mathtools}
\usepackage{enumerate}
\theoremstyle{plain} 
\newtheorem{thm}{Theorem}

\newtheorem{cor}[thm]{Corollary} 
\newtheorem{lem}[thm]{Lemma}

\theoremstyle{definition} 
\newtheorem{defn}{Definition}

\newtheorem{rem}[thm]{Remark} 
\newtheorem*{remk}{Remark}

\title {Signatures of topological branched covers}
\author{Christian Geske, Alexandra Kjuchukova, Julius L. Shaneson}
\thanks{This work was partially supported by NSF grant DMS 1821257 and NSF-RTG grant 1502553.}
\begin{document}
\maketitle

\begin{abstract}  

Let $X^4$ and $Y^4$ be smooth manifolds and $f: X\to Y$ a branched cover with branching set $B$. Classically, if $B$ is smoothly embedded in $Y$, the signature $\sigma(X)$ can be computed from data about $Y$, $B$ and the local degrees of $f$. When $f$ is an irregular dihedral cover and $B\subset Y$ smoothly embedded away from a cone singularity whose link is $K$, the second author gave a formula for the contribution $\Xi(K)$ to  $\sigma(X)$ resulting from the non-smooth point. 
We extend the above results to the case where $Y$ is a {\it topological} four-manifold and $B$ is locally flat, away from the possible singularity. Owing to the presence of non-locally-flat points on $B$, $X$ in this setting is a stratified pseudomanifold, and we use the Intersection Homology signature of $X$, $\sigma_{IH}(X)$.  For any knot $K$ whose determinant is not $\pm 1$, a homotopy ribbon obstruction is derived from $\Xi(K)$, providing a new technique to potentially detect slice knots that are not ribbon.  

\end{abstract}

\section{Introduction}

We give formulas for the signatures of branched covers in several contexts. First, suppose $B^{n-2} \subset Y^n$ is an inclusion of one topological manifold into another. 
Any connected unbranched cover of $Y-B$ can be uniquely extended to a branched cover $f: X \rightarrow Y$, a construction formalized in \cite{fox1957covering}.
If $B$ happens to be embedded locally flatly in $Y$, then $X$, too, is a topological manifold, and we can compute the signature of $X$.

\begin{thm}
\label{viro}
	Let $X$ and $Y$ be closed 4-dimensional topological manifolds, and let $f: X\to Y$ be an $n$-fold branched cover with branching set a closed, 
	locally flat surface $B$ embedded in $Y$. Assume that $X$ and $Y$ are compatibly oriented. Denote $A:=f^{-1}(B)$ and let $A_r\subset A$ be the union of components of branching index $r$; denote the normal Euler number of $A_r$ in $X$ by  $e(A_r)$. The following formula holds: 
	\begin{equation}
		 \sigma(X) = n \sigma(Y) - \sum_{r=2}^\infty \frac{r^2-1}{3} e(A_r).
	\end{equation}
\end{thm}

 This extends classical results on the signatures of smooth branched covers~\cite{{hirz1969signature},{viro1984signature}}. The conclusion of Theorem~\ref{viro} appears to have been taken for granted by many yet viewed with uncertainty by others. As far as we know, it has never appeared in print.

Our second signature formula applies to dihedral covers branched along surfaces embedded in the base with finitely many cone singularities. {For simplicity, we state the formula for the case of one singular point. The covering spaces we obtain are} 
  topological pseudomanifolds~\cite[Definition 4.1.1]{banagl2007topological}. Over a neighborhood of {each} cone singularity, the covering map is the cone on a branched cover in the usual sense. Here, we use $\sigma_{IH}(X)$, the intersection homology signature of $X$, which in the case of isolated singularities equals the Novikov signature of the manifold with boundary obtained by removing a small open neighborhood of the singular set. The definition and properties of irregular dihedral covers are recalled in Section~\ref{xi}. 

\begin{thm}\label{sashka}
Let $B \subset Y$ be the inclusion of a closed, surface with one cone singularity of type $K$ into a closed, oriented topological $4$-manifold.
Suppose also that $f: X \rightarrow Y$ is an irregular dihedral branched cover with branching set $B$, and which induces a $p$-coloring on the knot $K$.
Then:
\begin{equation}\label{eq-sigma}
	\sigma_{IH}(X) = p\sigma(Y)-\frac{p-1}{4}e(B)-\Xi_p(K).
\end{equation}
Here $e(B)$ denotes the self-intersection of $B$ in $Y$, and $\Xi_p(K)$ is an invariant of the knot $K$, together with its $p$-coloring, as introduced in~\cite{kjuchukova2018dihedral}.

\end{thm}

Here,  $\Xi_p(K)$ is an invariant of $K$ defined in~\cite{kjuchukova2018dihedral} for singularities that arise on dihedral covers between PL four-manifolds. By allowing the total space of the cover to be a topological pseudomanifold, we extend the definition of $\Xi_p$ to a much larger class of $p$-colorable knots $K$. The invariant $\Xi_p$ can be computed by a formula given in~\cite[Theorem 1.4 (1.3)]{kjuchukova2018dihedral}, which we now recall. Given a knot $K$ as above, let $V$ be a Seifert surface for $K$ and let $\kappa \subset V^\circ$ be a mod~$p$ characteristic knot for $K$ (as defined in~\cite{CS1984linking}), corresponding to the $p$-coloring of $K$ induced by the map $f$. Denote by $L_V$ the symmetrized linking form for $V$ and by $\sigma_{\zeta^i}$ the Tristram-Levine $\zeta^i$-signature, where $\zeta$ is a primitive $p^{\text{th}}$ root of unity.  Finally, let $W(K,\kappa)$ be the cobordism constructed in~\cite{CS1984linking} between the $p$-fold cyclic cover of $S^3$ branched along $\kappa$ and the $p$-fold irregular dihedral cover of $S^3$ branched along $K$ and determined by the $p$-coloring. We extend Theorem~1.4 of~\cite{kjuchukova2018dihedral} to show that, for all such $K$,

\begin{equation}
\label{xi}
	\Xi_p(K)=\dfrac{p^2-1}{6p}L_V(\kappa,\kappa)+\sigma(W(K,\kappa))+\sum_{i=1}^{p-1}\sigma_{\zeta^i}(\kappa).
\end{equation}

Using the above formula, $\Xi_p(K)$ can be evaluated directly from a $p$-colored diagram of $K$, without reference to the manifold $X$. An explicit algorithm for performing this computation is given in~\cite{cahnkjuchukova2018computing}. When $Y=S^4$, an alternative method for computing $\Xi_p(K)$ is to produce a trisection of $X$ from a colored diagram of $K$~\cite{cahnkjuchukova2017singbranchedcovers}.

{When the branching set has multiple singularities of types $K_1, \dots, K_n$, each of the knots $K_i$ contributes $-\Xi_p(K_i)$ to $\sigma(Y)$ in Equation~(\ref{eq-sigma}). This follows from the fact that Theorem~\ref{sashka} is proved by modifying the given branched cover in a neighborhood of the singular point. In the case of multiple singularities, this local procedure can be performed for all cone points simultaneously.}

Our third main result relates dihedral branched covers of $S^4$ to the Slice-Ribbon Conjecture of Fox in a new way. Specifically,  we use the invariant $\Xi_p(K)$ with $p$ square-free to obtain a homotopy ribbon obstruction for $K$. We show that for $K$ a $p$-colorable ribbon knot, the values of $\Xi_p(K)$ must fall within a bounded range. The analogous statement does not obviously hold for slice knots.

\begin{thm}
\label{inequality}
Assume $p > 1$ is odd and square-free.
Suppose $K \subset S^3$	 is a knot which admits $p$-colorings.
If $K$ is homotopy ribbon, then there is a $p$-coloring on $K$ for which:
\begin{align}
\label{ineq}
|\Xi_p(K)| \leq rk\, H_1(M)+\frac{p-1}{2}.  	
\end{align}
where $M \rightarrow S^3$ is the irregular dihedral cover branched along $K$ induced by the $p$-coloring.
\end{thm}

This result significantly expands the class of knots covered by the obstruction in~\cite{cahnkjuchukova2017singbranchedcovers}, which applies to slice knots with dihedral covers   
$S^3$. In contrast, Theorem~\ref{inequality} includes all slice knots that admit $p$-colorings for some square-free $p$. Any slice knot whose determinant is not $\pm 1$ is therefore covered by this theorem and could potentially be shown non-ribbon using the  $\Xi_p(K)$ invariant.   

\section{Signatures of branched covers between topological four-manifolds}	
\label{top}

This section is dedicated to the proof of Theorem~\ref{viro}. 
For the convenience of the reader, we recall that for $B^2\subset Y^4$ a locally flat submanifold as in the statement of the theorem, and the normal Euler number is defined as follows: If $B$ happens to be orientable, then a choice of fundamental class determines an element $[B] \in H_2(Y)$ as well as its Poincar\'{e} dual $\nu_B \in H^2(Y)$. Then $e(B)$ is the evaluation $\nu_B[B] \in \mathbb{Z}$.
If $B$ is not orientable, select a neighborhood $T$ of $B$ such that $B \hookrightarrow T$ induces an isomorphism of fundamental groups. 
	Then the oriented double cover $\widetilde{B} \rightarrow B$ induces a double cover $\widetilde{T} \rightarrow T$.
	A choice of fundamental class for $\widetilde{B}$ determines an element $[\widetilde{B}] \in H_2(\widetilde{T})$ and its Poincar\'{e} dual $\nu_{\widetilde{B}} \in H^2_c(\widetilde{T})$.
	The normal euler number is $\frac{1}{2}\nu_{\widetilde{B}}[\widetilde{B}] \in \mathbb{Z}$ (see, for example,~\cite[Section 3]{kamada1989nonorientable}).

We will need the following lemmas.  

\begin{lem}\label{connectsmooth}
Suppose $Y$ is an oriented topological $4$-manifold. There exists an oriented $4$-manifold $Z$ such that $Y \# Z$ admits a smooth structure and  $\sigma(Y\#Z) = \sigma(Y)$ or $2\sigma(Y)$.
\end{lem}
\begin{proof}
The first obstruction to $Y$ admitting a smooth structure is the Kirby-Siebenmann invariant, $ks(Y)$. Since this is a $\mathbb{Z}/2\mathbb{Z}$-valued invariant which is additive with respect to the connected sum operation, if it does not vanish for $Y$, it does for $Y\#Y$. By further taking the connected sum with finitely many copies of $S^2 \times S^2$, we obtain a four-manifold which admits a smooth structure~\cite{lashof1971smoothing}.
Put together, we have found an oriented $4$-manifold $Z$ such that the connect sum $Y \# Z$ admits a smooth structure. Here, if $Z$ is non-empty, we have $Z\cong \#k(S^2 \times S^2)$ or $Z\cong Y\#k(S^2 \times S^2)$. Since $\sigma(S^2 \times S^2)=0$, by Novikov additivity we have that $\sigma(Y \# Z)=\sigma(Y)$ or $2\sigma(Y)$.
\end{proof}

Every locally flat embedding of a surface $B$ into a topological $4$-manifold $Y$ comes with a normal bundle, as follows from \cite[Section 9.3]{freedman1990topology}.
Here by a {\it normal bundle} of the embedding $B \hookrightarrow Y$, we mean a two-dimensional vector bundle $E \rightarrow B$, together with an embedding $E \hookrightarrow Y$, which when restricted to the zero-section of $E$ is the given embedding of $B$. 
Observe that the embedding of $E$ into $Y$ is open by invariance of domain.

We need to consider the possible smooth structures on a vector bundle $E$ over a surface $B$.
Observe that $E$ admits at least one smooth structure: select a smooth structure on $B$, then $E$ can be endowed with a smooth structure as a vector bundle over a smooth manifold. On the other hand, we will need to consider the smooth structure induced on $E$ when regarded as an open submanifold of the smooth manifold $Y\#Z$ of Lemma~\ref{connectsmooth}. 

\begin{defn}
	Two smooth structures $\zeta$ and $\zeta'$ on a manifold $W$ are {\it concordant} {if} there is a smooth structure on $W \times [0,1]$ which restricts to $\zeta$ on $W \times 0$ and to $\zeta'$ on $W \times 1$.  We call this smooth structure on $W \times [0,1]$ a \textit{concordance}. 	
\end{defn}

\begin{lem}\label{concordant}
Suppose $E \rightarrow B$ is a two-dimensional vector bundle over a surface $B$.
Any two smooth structures on $E$ are concordant.
\end{lem}
\begin{proof}
By \cite[Theorem 8.7B]{freedman1990topology}, concordance classes of smooth structures on $E$ are classified by elements $H^3(E;\mathbb{Z}_2)$. Since $E$ has the homotopy type of a surface, $H^3(E;\mathbb{Z}_2)$ is trivial. 
\end{proof}

\begin{lem}\label{smoothcover}
Let $f: X \rightarrow Y$ be a branched cover between topological manifolds such that:
\begin{enumerate}[(i)]
\item $Y$ is smooth.
\item the branching set $B \subset Y$ is a smooth codimension two submanifold of $Y$.
\end{enumerate}
Then $X$ admits a {unique} smooth structure for which $f$ is a smooth branched cover in the sense of \cite{viro1984signature}.
\end{lem}
\begin{proof}
The statement follows from \cite{durfee1975periodicity} Proposition~1.1.
\end{proof}

Viro's original formula for the signatures of branched covers~ \cite{viro1984signature} applies in higher dimensions as well, and requires the following notion of self-intersections of manifolds.  
If $A$ is a closed smooth submanifold of an oriented smooth manifold $Y$, then we can slightly isotope the inclusion $A \hookrightarrow X$ to obtain a smooth embedding $f: A \rightarrow X$ whose image  $\widetilde{A}$ is transverse to $A$ in $X$.
Define the {\it self-intersection $A^2$ of $A$ in $X$} to be the intersection $\widetilde{A} \cap A$.
The manifold $A^2$ can be equipped with a canonical orientation which does not depend up to oriented cobordism on the choice of small isotopy (see \cite[Section 1.2]{viro1984signature}).
In particular, we may unambiguosly discuss the signature of $A^2$.

\iffalse
The manifold $A^2$ can be equipped with an orientation in the following way:

Select a Riemannian metric on $X$, using which we will consider various normal bundles and projections.
We use $T(-)$ for the tangent bundle of $(-)$, and use $\nu_{(\bullet)}(-)$ for the normal bundle of $(-)$ in $(\bullet)$.
Consider now the normal projection $TX|_A \rightarrow \nu_X A$.
By transversality of $A$ and $\widetilde{A}$, this induces an isomoprhism:
\begin{align*}
    \nu_{\widetilde{A}}A^2 \xrightarrow{\cong} (\nu_X A)|_{A^2}
\end{align*}
and, assuming the homotopy is small enough, also induces an isomorphism:
\begin{align*}
    (\nu_X \widetilde{A})|_{A^2} \xrightarrow{\cong} (\nu_X A)|_{A^2}.
\end{align*}
So we have:
\begin{align*}
    \nu_{X} A^2 = (\nu_{X} \widetilde{A})|_{A^2} \oplus \nu_{\widetilde{A}} A^2 \cong (\nu_X A)|_{A^2} \oplus (\nu_X A)|_{A^2}.
\end{align*}
This bundle has a canonical orientation (as a direct sum of two copies of the same bundle).
Since:
\begin{align*}
   (TX)|_{A^2} = \nu_X A^2 \oplus TA^2 
\end{align*}
and the first two bundles are oriented, so is $TA^2$, and therefore the manifold $A^2$.
This orientation depends only on the orientation of $X$.

The oriented manifold $A^2$ constructed in this way does not depend up to oriented cobordism on the choice of small isotopy $f$  (see \cite[Section 1.2]{viro1984signature}).
In particular, we may unambiguosly discuss the signature of $A^2$.
\fi

\begin{lem}\label{selfintersection}
Suppose $A$ is a closed smooth submanifold of an oriented smooth manifold $X$, and $Z$ is a closed smooth, oriented manifold.
Consider $A \subset X$ and $A \times Z \subset X \times Z$.
The self-intersections $A^2 \times Z$ and $(A \times Z)^2$ determine the same oriented cobordism class. 
In particular, we have the following relation among signatures of self-intersections:
\begin{align*}
    \sigma\left(\left(A \times Z\right)^2\right) = \sigma(A^2) \cdot \sigma(Z).
\end{align*}
\end{lem}
\begin{proof}
If $\widetilde{A}\subset X$ is a transverse copy of $A$, then $\widetilde{A} \times Z\subset X\times Z$ is a transverse copy of $A\times Z$. The intersection of $A \times Z$ and $\widetilde{A} \times Z$ is $A^2 \times Z$.
\end{proof}

\begin{proof}[Proof of Theorem~\ref{viro}]
Our strategy is to modify the manifolds and covering maps involved until we produce a smooth branched cover in the sense of~\cite{viro1984signature} whose signature we can {\it (i)} express in terms of the signature of $X$; {\it (ii)} compute from the original map $f$ and its branching set.

\textit{Step 1. Reduce to the case where the base $Y$ admits a smooth structure.} 

Fix a branched cover $f: X^4\to Y^4$ as in the statement of the theorem, with $Y$ potentially non-smooth. By Lemma~\ref{connectsmooth}, there exists an oriented $4$-manifold $Z$ such that $Y \# Z$ admits a smooth structure.
Perform the connected sum operation away from the branching set so that the $n$-sheeted branched cover $f: X \rightarrow Y$ induces an $n$-sheeted branched cover $f': X' \rightarrow Y'$ whose downstairs (resp.\ upstairs) branching set is also $B$ (resp.\ $A$) and such that $X' \cong X \# n Z$ is the connect sum of $X$ with $n$ copies of $Z$. 

By Novikov Additivity,
the signature of $Y'$ is $\sigma(Y)+\sigma(Z)$ and the signature of $X'$ is $\sigma(X)+n\sigma(Z)$.
The normal Euler number of $A_r$, being a local invariant around $A_r$, remains constant whether we consider $A_r$ as a subset of $X$ or as a subset of $X'$.

Let's assume Theorem~\ref{viro} has been proven for $f': X' \rightarrow Y'$.
Combining the signature formula with the observations of the previous paragraph we find:
\begin{align*}
    \sigma(X)+n\sigma(Z) = n \sigma(Y) + n \sigma(Z) - \sum_{r=2}^\infty \frac{r^2-1}{3} \cdot e(A_r)
\end{align*}
from which Theorem~\ref{viro} is shown to hold for $f: X \rightarrow Y$.

Thus, it suffices to prove the theorem for the case where the base manifold $Y$ admits a smooth structure.

\textit{Step 2. Provided that $Y$ admits a smooth structure, relate the given branched cover to a smooth branched cover.}

Fix a smooth structure $\zeta$ on $Y$ and let $E$ be the total space of the normal bundle of $B$ in $Y$.
Select also a smooth structure on the surface $B$. Being a vector bundle over the smooth manifold $B$, $E$ inherits a smooth structure $\zeta'$ .
This smooth structure on $E$ is good for our purposes, because it renders $B$  a {smoothly} embedded submanifold of $E$.
Unfortunately, this smooth structure is not necessarily extendable to a smooth structure on $Y$.

But $E$ admits another smooth structure: the restriction $\zeta|_E$ of the already existing smooth structure on $Y$.
Lemma~\ref{concordant} shows that $\zeta|_E$ and $\zeta'$ are concordant.
Let $\Gamma$ be such a concordance of smooth structures.
In other words, $\Gamma$ is a smooth structure on $E \times [0,1]$ which restricts to $\zeta|_E$ on $E \times \{0\}$ and to $\zeta'$ on $E \times \{1\}$. We would like to extend this concordance to one on $Y$.

To that end, we consider the product of all our spaces with $\mathbb{C} P^2$, which will allow us to apply smoothing theorems that work only in higher dimensions.
This gives the $n$-sheeted branched cover:
\begin{align*}
f \times id: X \times \mathbb{C} P^2 \rightarrow Y \times \mathbb{C} P^2 \end{align*}
with downstairs branching set $B \times \mathbb{C} P^2$ and upstairs branching set $A \times \mathbb{C} P^2$.
Let $\alpha$ be the standard smooth structure on $\mathbb{C} P^2$.
The base $Y \times \mathbb{C} P^2$ admits product smooth structure $\zeta \times \alpha$.
The restriction of this smooth structure to $E \times \mathbb{C} P^2$ is:
\begin{align*}
    (\zeta \times \alpha)|_{E \times \mathbb{C} P^2} = \zeta|_E \times \alpha.
\end{align*}
We have also the product smooth structure $\Gamma \times \alpha$ on $E \times [0,1] \times \mathbb{C}P^2$.
The restriction of this smooth structure to $E \times \{0\} \times \mathbb{C} P^2$ is $\zeta|_E \times \alpha$ and the restriction to $E \times \{1\} \times \mathbb{C} P^2$ is $\zeta' \times \alpha$. 
Therefore the smooth structures $(\zeta \times \alpha)|_{E \times \mathbb{C}P^2}$ and $\zeta' \times \alpha$ on $E \times \mathbb{C}P^2$ are concordant.

By \cite[Concordance Extension Theorem]{kirby1977foundational} this concordance of smooth structures on $E \times \mathbb{C} P^2$ {extends} to a concordance of smooth structures on the enlarged product $Y \times \mathbb{C} P^2$, i.e.\ a concordance which restricts to $\zeta \times \alpha$ on $Y \times  \{0\} \times \mathbb{C} P^2$ and restricts to $\zeta' \times \alpha$ on $E \times  \{1\} \times \mathbb{C} P^2$.
In particular, the restriction of the extended concordance to all of $Y \times  \{1\} \times \mathbb{C} P^2$ 
gives a smooth structure on $Y \times \mathbb{C} P^2$ that extends the smooth structure $\zeta' \times \alpha$ on $E \times \mathbb{C} P^2$.
Note that we have not changed the topological structure of $Y \times \mathbb{C} P^2$ in any way, but have merely found for it a particular smooth structure. 

Equip $Y \times \mathbb{C} P^2$ with the smooth structure extending $\zeta' \times \alpha$ whose existence we just established.
For this structure, the branching set $B \times \mathbb{C} P^2$ (which again has not changed topologically) becomes a smoothly embedded submanifold.
By Lemma~\ref{smoothcover}, the branched cover $X \times \mathbb{C} P^2 \rightarrow Y \times \mathbb{C} P^2$ is smooth.
Viro's original formula \cite{viro1984signature} now applies.

\textit{Step 3. Apply Viro's formula to the smooth branched cover $f\times id: X \times \mathbb{C} P^2 \rightarrow Y \times \mathbb{C} P^2$.}

Recall that the upstairs branching set of the cover $f\times id$ is  $A \times \mathbb{C} P^2$.

Using multiplicativity of the signature, and $\sigma(\mathbb{C} P^2) = 1$, the formula yields:
\begin{align*}
    \sigma(X) = n\sigma(Y) + \mbox{local contributions near $A \times \mathbb{C} P^2$.}
\end{align*}
These local contributions take into account iterated self-intersections of $A_r \times \mathbb{C} P^2$ in $X \times \mathbb{C} P^2$, where we recall that $A_r$ is the union of those components of $A$ with branching index $r \geq 2$.
By Lemma~\ref{selfintersection}, the self-intersection $(A_r \times \mathbb{C} P^2)^2$ is $A_r^2 \times \mathbb{C} P^2$.
Because $A_r$ is a closed two-dimensional submanifold of a four-manifold, $A_r^2$ is a finite set of points in $X$.
So the space $A_r^2 \times \mathbb{C} P^2$ has empty self-intersection, and all further iterated self-intersections will vanish.
Therefore, the formula gives:
\begin{align*}
    \sigma(X) = n\sigma(Y) - \sum_{r = 2}^\infty \frac{r^2-1}{3r}\cdot \sigma(A_r^2 \times \mathbb{C} P^2) = n\sigma(Y)-\sum_{r = 2}^\infty \frac{r^2-1}{3r}\cdot \sigma(A_r^2).
\end{align*}
But $\sigma(A_r^2)$ counts the number of points of self-intersection with sign, and is exactly the normal Euler number $e(A_r)$ of $A_r$ in $X$.
\end{proof}

\section{Signatures of dihedral branched covers}	
\label{xi}	

In this section, we prove Theorem~\ref{sashka}. For the rest of this paper we focus on {\it irregular dihedral} branched covers, whose definition is recalled below. We denote the dihedral group of order $2p$ by $D_p$ and, throughout this article, we assume that $p$ is odd.

Let $B \subset Y$ be a codimension-two inclusion of topological manifolds.
Let $X \rightarrow Y$ be a connected branched cover with branching set $B$.
Recall that such a branched cover is determined by a connected unbranched cover of $Y-B$, which is in turn determined by the conjugacy class of a subgroup of $\pi_1(Y-B)$.

\begin{defn}
\label{dihedral}
We say that the branched cover $X \rightarrow Y$ is {\it irregular dihedral} if it corresponds to a conjugacy class $\rho^{-1}(\mathbb{Z}_2)$ where $\rho: \pi_1(Y-B) \twoheadrightarrow D_p$ is some surjection and $\mathbb{Z}_2$ is the conjugacy class of elements of order $2$ in $D_p$.   
\end{defn}

Note that surjectivity of $\rho$ implies that an irregular dihedral cover is connected.

\begin{defn}\label{conesing}
We say that an inclusion $B \subset Y$ of a surface into a topological $4$-manifold has a {\it cone singularity of type the knot $K \subset S^3$} if there exists distinguished $b_0 \in B$ such that:
\begin{itemize}
\item The inclusion $B\backslash\{b_0\} \hookrightarrow Y$ is locally flat.
\item $b_0$ has a neighborhood $V$ in $Y$ for which there is a homeomorphism of triples:
$$(V, V \cap B, b_0) \cong (cS^3, cK, *)$$
where $c$ denotes the cone operation and $*$ the cone point.
\end{itemize}
\end{defn}

We are interested in irregular dihedral covers $f: X \rightarrow Y$ over $4$-manifolds. We will assume that the branching set $B$ has a cone singularity whose type is a knot $K$.
Let $\rho: \pi_1(Y-B) \twoheadrightarrow D_p$ be the surjective homomorphism associated to the branched cover.
By identifying $(V, V \cap B, b_0)$ with $(cS^3, cK, *)$, we can consider the composition:
$$\pi_1(S^3\backslash K) \xrightarrow{i_*} \pi_1(Y-B) \xrightarrow{\rho} D_p.$$
When this composition is a surjection, we say that the cover $f$ \textit{induces a $p$-coloring on the knot $K$} or, equivalently, that the singularity at $b_0$ is \textit{normal}.
Again, this is equivalent to the associated irregular dihedral cover $M \rightarrow S^3$ branched along $K$ being connected. 

\begin{remk}\label{irregprop} We state a few basic properties of irregular dihedral branched covers.
Suppose $B \subset Y$ is the inclusion of a connected surface with a cone singularity of type $K$ into a 4-manifold $Y$.
Suppose also that $f: X \rightarrow Y$ is an irregular dihedral branched cover with branching set $B$, such that $f$ induces a $p$-coloring on the knot $K$ in the above sense.
Let $M \rightarrow S^3$ denote the induced irregular dihedral cover branched along $K$.	
It follows that:
\begin{itemize}
\item $f^{-1}(S^3 - K)$ is connected and $f^{-1}(S^3) = M$.
\item The restriction $f|: f^{-1}(B\backslash\{b_0\}) \rightarrow B\backslash\{b_0\}$ is an unbranched covering map with $\frac{1}{2}(p+1)$ sheets.
\item For any $b \in B\backslash\{b_0\}$, $f^{-1}(b)$ contains $\frac{1}{2}(p-1)$ points of branching index $2$ and one point of branching index $1$.
\end{itemize}
These properties can be verified via the local parametrization description for branched covers and with the aid of \cite[Proposition 11.2]{massey1967algebraic}.
\end{remk}

In our generalization of Kjuchukova's signature formula, we allow the branched covering space to be a topological pseudomanifold, as defined for example in \cite[Definition 4.1.1]{banagl2007topological}.
We begin by proving:

\begin{lem}\label{pseudomanifold}
	Let $B \subset Y$ be the inclusion of a surface with a cone singularity of type $K$ into a topological $4$-manifold.
	Assume that $f: X \rightarrow Y$ is an irregular dihedral branched cover with branching set $B$ and that $f$ induces a $p$-coloring on the knot $K$.
	Then $X$ has the structure of a pseudomanifold.
\end{lem}
\begin{proof}

Denote by $b_0$  the singularity of the embedding of $B$ in $Y$, and let $f|: M \rightarrow S^3$ be the induced (connected,  by Remark~\ref{irregprop}) irregular dihedral cover with branching set $K$. 
Then $f^{-1}(b_0)$ consists of a single point $b'_0$  that has conic neighborhood $cM$ in $X$.
In other words:
$$X \supset \{b_0'\}$$
is a pseudomanifold stratification, where the link of $b'_0$ in $X$ is $M$.
\end{proof}

If the manifold $Y$ in Lemma~\ref{pseudomanifold} is oriented, then $X$ inherits a pseudomanifold orientation.
In our case this is merely a manifold orientation on the complement of the singular set $X-\{b'_0\}$ as for example defined in \cite[Definition 8.1.5]{friedman2014singular}\footnote{{Definition and theorem numbering from reference \cite{friedman2014singular} is based on the July 24, 2019 preprint.}} in the topological setting.
If $Y$ happens to be PL, an orientation can also be interpreted as a coherent orientation of all $4$-simplices in a triangulation of $Y$ (see \cite[I.2.3]{borel2009intersection}).
If $Y$ is in addition closed, then so is $X$ and we can discuss the intersection homology signature $\sigma_{IH}(X)$ of $X$ (again see \cite[Definition 9.3.9]{friedman2014singular}, which applies in the topological setting).
We are now in a position to prove Theorem~\ref{sashka}, our generalization of Kjuchukova's signature formula in this context.

\textit{Remark.} As noted earlier, since $X$ has an isolated singularity, the intersection homology signature $\sigma_{IH}(X)$ is equal to the Novikov signature of the manifold with boundary obtained by removing from $X$ a small open neighborhood of the singularity.
Accordingly, our arguments can be recast in this language.

\begin{proof}[Proof of Theorem~\ref{sashka}]
By Lemma~\ref{pseudomanifold} and the comments following, $X$ is a closed topological pseudomanifold, and inherits an orientation from the orientation on $Y$.
Therefore, we can discuss the intersection homology signature of $X$.

Let $b_0 \in B$ be the singular point on $B$ and let $V$ be a neighborhood of $b_0$ in $Y$ as in Definition~\ref{conesing}.
Denote by $V^\circ$ the interior of $V$.	
By the proof of Lemma~\ref{pseudomanifold}, we can shrink $V$ if necessary to ensure that $M = f^{-1}(S^3)$ has a collar neighborhood in $X$, in which case $X$ can be decomposed into the union of oriented pseudomanifolds with boundary:
\begin{align*}
X = f^{-1}(V) \cup f^{-1}(Y\backslash V^\circ),	
\end{align*}
where $f^{-1}(V)$ is the cone on the topological manifold $M$.
The Novikov intersection homology signature \cite[Definition 9.3.11]{friedman2014singular} of the cone on a three-dimensional topological manifold is always zero. Indeed, it is 
the signature of
 a pairing on a subspace of $IH^{\bar{n}}_2(cM; \mathbb{Q})$, where $\bar{n}$ denotes the upper-middle perversity. But $IH^{\bar{n}}_2(cM; \mathbb{Q})$ vanishes by the cone formula \cite[Theorem 4.2.1]{friedman2014singular}.   

Combining the above with Novikov additivity~\cite[Theorem 9.3.22]{friedman2014singular}, we obtain the equalities:
\begin{align*}
\sigma_{IH}(X) = \sigma_{IH} \left(f^{-1}(Y\backslash V^\circ)\right) = \sigma \left(f^{-1}(Y\backslash V^\circ)\right),
\end{align*}
where the $IH$ subscript at the end is dropped because $f^{-1}(Y\backslash V^\circ)$ is a manifold with boundary.

The signature of $f^{-1}(Y\backslash V^\circ)$ was computed in~\cite[Proof of Theorem 1.4]{kjuchukova2018dihedral}. That is,
\begin{align}
\label{signatureformula}
\sigma\left(f^{-1}(Y\backslash V^\circ)\right) = p\sigma(Y) - \frac{p-1}{4} e(B) - \Xi_p(K),	
\end{align}
under the additional assumptions that $f^{-1}(S^3) \cong S^3$ and that all spaces and maps be PL. We will outline Kjuchukova's proof of the above in order to highlight the points where these assumptions play a role and to explain that they can be dropped using our Theorem~\ref{viro}.

The only role of the assumption that $f^{-1}(S^3) \cong S^3$ is to obtain a total space $X$ that is a PL manifold. Since we allow the total space $X$ to be a pseudomanifold by working with $\sigma_{IH}(X)$, there is no reason for us to retain this hypothesis.

The idea of the proof of~\cite[Theorem 1.4]{kjuchukova2018dihedral} is to geometrically resolve the singularity on the branching set of the given branched cover $f: X\to Y$, so that Viro's signature formula for smooth covers applies. In the process, one also keeps track of the effect that the resolution of singularity has on the signatures involved. 

To this end, the first step is to construct a new manifold $Z$ (in the notation of~\cite{kjuchukova2018dihedral}, $Z$ is $W(\alpha, \beta)\cup Q$) that also has $f^{-1}(S^3)$ as its boundary, with the property that the boundary union $T:=f^{-1}(Y\backslash V^\circ)\cup Z$ is also a $p$-fold branched cover of $Y$. 

The branching set $B'$ of the new cover $T\to Y$ differs from $B$ in a neighborhood of the singular point and is no longer an embedded surface. Instead, $B'$ is a two-dimensional PL subcomplex which contains a curve of non-manifold points.  
The next step is to simultaneously modify $Y$ and $T$ via a doubling construction that removes the non-manifold points on $B'$.  The result is yet another branched cover between PL manifolds, this time with a locally flat embedded surface for its branching set.
 
It is then possible to relate in a straightforward manner the signature of this final branched cover to the signature of $f^{-1}(Y\backslash V^\circ)$.
Viro's formula is used to find the signature of the branched cover thus constructed. Furthermore, the effect which the singularity resolution procedure has on the signature of the covering manifold is expressed in terms of the knot $K$ in Equation~(1.3) of~\cite[Theorem 1.4]{kjuchukova2018dihedral}. This accounts for the term $-\Xi_p(K)$ in Equation~\ref{signatureformula}. Note that applying Viro's formula is the step in which the PL assumption is necessary (using the fact that PL implies smooth in four dimensions). We can now remove the PL assumption thanks to Theorem~\ref{viro}. This leads us to:
\begin{align*}
\sigma\left(f^{-1}(Y\backslash V^\circ)\right) = p\sigma(Y) - \frac{p-1}{4} e(B) - \Xi_p(K).	
\end{align*}

Because $\sigma_{IH}(X) = \sigma\left(f^{-1}(Y\backslash V^\circ)\right)$, the theorem follows.
\end{proof}

\section{A homotopy ribbon obstruction}	
\label{ribbon}	

This section is dedicated to the proof of Theorem~\ref{inequality}. Recall that a slice knot $K \subset S^3$ is said to be \textit{topologically (resp. smoothly) homotopy ribbon} if there exists a locally flat (resp. smooth) slice disk $D \subset B^4$ for $K$ such that the homomorphism induced by inclusion,
$$i_*: \pi_1(S^3\backslash K) \rightarrow \pi_1(B^4-D),$$
is surjective.
We call $D$ a \textit{homotopy ribbon disk} for $K$. 
Theorem~\ref{inequality} shows that the knot invariant $\Xi_p(K)$ can be used to obtain an obstruction to the existence of such a disk for a $p$-colorable slice knot~$K$.
Our proof makes use of the following preliminary results.

\begin{lem}\label{diskext}
Let $K \subset S^3$ be a $p$-colorable slice knot, where $p>1$ is an odd square-free integer. Given any topologically slice disk $D$ for $K$, some $p$-coloring of $K$ extends over $D$. Equivalently, there is a $p$-fold irregular dihedral cover of $B^4$ with branching set $D$ that induces a $p$-coloring of $K$.

Suppose $W \rightarrow B^4$ is one such irregular dihedral cover of $B^4$ branched along $D$. Consider the induced $p$-coloring on $K$ and let $M \rightarrow S^3$ denote the associated irregular dihedral cover branched along $K$.
If $D$ is homotopy ribbon, then $rk\, H_1(W) \leq rk\, H_1(M)$.  
\end{lem}
\begin{proof}
The first paragraph is a classical result, essentially a consequence of the Fox-Milnor condition on the Alexander polynomials of slice knots~\cite{fox1966singularities}. We briefly sketch the argument. Recall that a $p$-fold irregular dihedral cover of $S^3$ branched along $K$ exists if and only if the two-fold branched cover $\Sigma_2$ of $K$ admits a $p$-fold cyclic unbranched cover, as explained in~\cite{CS1984linking}. Additionally, a $p$-fold unbranched cover of $\Sigma_2$ exists if and only if $H_1(\Sigma_2; \mathbb{Z})$ surjects onto  $\mathbb{Z}/p\mathbb{Z}$. Indeed, every surjection $\pi_1(S^3-K) \twoheadrightarrow D_p = \mathbb{Z}_2 \ltimes \mathbb{Z}_p$ factors through $\mathbb{Z}_2 \ltimes H_1(\Sigma_2;\mathbb{Z})$ where $\mathbb{Z}_2$ acts on $H_1(\Sigma_2; \mathbb{Z})$ by translation (which happens also to be inversion), so that a surjection $H_1(\Sigma_2;\mathbb{Z}) \twoheadrightarrow \mathbb{Z}_p$ induces a surjection $\pi_1(S^3-K) \twoheadrightarrow D_p$ (see \cite[Proposition 14.3 and Remark before Theorem 8.21]{burde2013knots}). For $p$ square-free, this is equivalent to saying that $p$ divides $|H_1(\Sigma_2; \mathbb{Z})|$. The analogous statements hold regarding the existence of a dihedral cover of $B^4$ branched along $D$. We now use~\cite[Lemma 3]{cass-gor1986cobordism} to conclude that, for $p$ square-free, $\Sigma_2$ admits a $p$-fold cyclic unbranched cover if and only if the double cover of $B^4$ branched along $D$ does. Put differently, some $p$-coloring of $K$ extends over $D$, and the induced dihedral branched cover of $D$ satisfies the conclusion of the lemma.

The proof of~\cite[Lemma 3.3]{kjuchukova2018dihedral} can be applied to conclude the second paragraph, since, if $D$ is homotopy ribbon, it follows that $i_\ast: \pi_1(M) \twoheadrightarrow \pi_1(W)$. 
\end{proof}

\begin{defn}
We say that a $p$-coloring of a knot $K$ is \textit{s-extendible} if it extends over some slice disk $D\subset B^4$ for $K$.
We say the coloring is \textit{hr-extendible} if it extends over some homotopy ribbon disk $D\subset B^4$ for $K$. 	
\end{defn}

\begin{rem}
\label{metab}
Let $p > 1$ be odd and square-free.
To determine whether a $p$-coloring of a knot $K$ is s-extendible, one may investigate the associated surjection $\overline{\rho}: H_1(\Sigma_2) \twoheadrightarrow \mathbb{Z}_p$.
A slice disk $D\subset B^4$ induces a metabolizer $H_D < H_1(\Sigma_2)$ with respect to the linking form by \cite[Theorem 2]{cass-gor1986cobordism}.
A necessary condition for the coloring to extend over $D$ is that $\overline{\rho}$ vanishes on $H_D$.
If $D$ is homotopy ribbon, then this condition is also sufficient.
\end{rem}

\begin{lem}\label{existence}
Assume $p > 1$ is odd and square-free.
Suppose $K \subset S^3$ is a slice knot which admits a $p$-coloring. 
Then there exists a $p$-fold irregular dihedral branched cover $X \rightarrow S^4$ whose branching set is an embedded $S^2$ with a normal cone singularity of type $K$. 

If $K$ is homotopy ribbon, then $f: X \rightarrow S^4$ can be chosen such that, if $M \rightarrow S^3$ denotes the induced irregular dihedral cover branched along $K$, we have $rk\, IH^{\bar{n}}_1(X) \leq rk\, H_1(M)$.
Here, $\bar{n}$ denotes the upper-middle perversity function.
\end{lem}

\begin{proof}
Since $K$ is slice, by Lemma~\ref{diskext} there exists an irregular dihedral cover $f: W \rightarrow B^4$ branched along a slice disk $D$ for $K$.
Its restriction to the boundary $\partial W = M$ is an irregular dihedral cover $f|: M \rightarrow S^3$ branched along $K$.
As in \cite[Proposition 3.4]{kjuchukova2018dihedral}, the map:
$$W \cup_M cM \rightarrow B^4 \cup_{S^3} cS^3$$ 
will serve as the desired irregular dihedral branched cover $X\to S^4$ with branching set $S^2 = D \cup_{\partial D} c(\partial D)$. The singularity on this sphere is the cone point and it is normal by construction. 

To prove the second paragraph of the Lemma, observe that the inclusion $M \hookrightarrow cM$ induces an isomorphism $H_1(M) \cong IH^{\bar{n}}_1(M)$ by the cone formula \cite[Theorem 4.2.1]{friedman2014singular}.
Since $M$ is moreover connected, the Mayer-Vietoris sequence for intersection homology \cite[Theorem 4.4.19]{friedman2014singular} shows that there is a short exact sequence $H_1(M) \hookrightarrow H_1(W) \oplus IH^{\bar{n}}_1(cM) \twoheadrightarrow IH^{\bar{n}}_1(X)$.
Therefore, $rk\, IH^{\bar{n}}_1(X) = rk\, H_1(W)$. 
Now apply the last statement of Lemma~\ref{diskext}.
\end{proof}

Next, we use Lemma~\ref{existence} to calculate the intersection homology Euler characteristic and signature of the cover $X$.

\begin{lem}\label{calc}
Assume $p > 1$ is odd and square-free.
Suppose $K \subset S^3$ is a knot which admits $p$-colorings. 
Let $f: X \rightarrow S^4$ be as in Lemma~\ref{existence}, and let $M \rightarrow S^3$ be the induced irregular dihedral cover of $S^3$ branched along $K$. 
Then the (upper-middle perversity $\bar{n}$) intersection homology Euler characteristic of $X$ is:
$$I\chi(X)= 1-rk\, H_1(M) + \frac{p+1}{2}$$
and its intersection homology signature is:
$$\sigma_{IH}(X) = - \Xi_p(K)$$
where $\Xi_p(K)$ is the invariant associated to the induced $p$-coloring on $K$.	
\end{lem}
\begin{proof}

The signature calculation is a straightforward application of Theorem~\ref{viro}, since $\sigma(S^4) = 0$ and $e(S^2)=0$ as well.

Adopting the notation used in the Proof of Lemma~\ref{existence}, the branched cover $f: X\to S^4$ can be rewritten as:
$$f: W \cup_M cM \rightarrow B^4 \cup_{S^3} cS^3.$$
Therefore, we have the following equality involving intersection homology Euler characterstics:
$$I\chi(X) = I\chi(W)+I\chi\left(cM\right) - I\chi(M).$$

We now calculate the component Euler characteristics.
Since $M$ is a manifold, its intersection homology Euler characteristic is equal to its usual Euler characteristic.
But $M$ is moreover closed and odd-dimensional.
Therefore,
$$I\chi(M) = \chi(M) = 0.$$

By the cone formula \cite[Theorem 4.2.1]{friedman2014singular}, the intersection homology of $cM$ is given by:
$$IH^{\bar{n}}_i(cM) = \left\{\begin{array}{ll}
H_i(M) & \mbox{if}\ i \leq 1\\
0 & \mbox{otherwise}.
\end{array}\right.
$$
Hence,
$$I\chi(cM) = 1 - rk\,H_1(M).$$

Lastly let's consider the manifold $W$, which is a $p$-fold irregular dihedral cover of $B^4$ branched along slice disk $D$.
Because $W$ is a manifold, we have again $I\chi(W) = \chi(W)$.

Let $A \subset W$ denote the upstairs branching set of the map $W\to B^4$ so that $\chi(W) = \chi(A) + \chi(W\backslash A)$, where we have additivity because $(A,\partial A) \subset (W,\partial W)$ is an inclusion of even-dimensional manifolds. 
Since the map $A \rightarrow D$ is an unbranched cover with $\frac{p+1}{2}$  sheets by Remark~\ref{irregprop} and $W\backslash A \rightarrow B^4\backslash D$ is a $p$-sheeted unbranched cover, we find:
$$I\chi(W) = \frac{p+1}{2} \chi(D) + p \left[\chi(B^4) - \chi(D)\right] = \frac{p+1}{2}.$$

Putting everything together, we conclude that, as claimed,:
$$I\chi(X) = 1 - rk\, H_1(M) + \frac{p+1}{2}.$$
\end{proof}

\begin{proof}[Proof of Theorem~\ref{inequality}]

On the one hand, using universal coefficients, and Poincar\'{e} duality \cite[Theorem 8.2.4]{friedman2014singular} for oriented, closed pseudomanifolds (and that the singular point of $X$ has even codimension) we have:
$$I\chi(X) = 2 - 2\cdot rk\, IH^{\bar{n}}_1(X) + rk\, IH^{\bar{n}}_2(X)$$ 
where $\bar{n}$ denotes the upper-middle perversity function.	
This in turn by Lemma~\ref{existence} is greater than or equal to $2 - 2 \cdot rk\, H_1(M)+ rk\, IH^{\bar{n}}_2(X)$.
On the other hand, by Lemma~\ref{calc}:
$$I\chi(X) = 1 - rk\, H_1(M) + \frac{p+1}{2}.$$
Comparing the two: 
\begin{align*}
	1 - rk\, H_1(M) + \frac{p+1}{2} \geq 2-2\cdot rk\, H_1(M) + rk\, IH^{\bar{n}}_2(X)  \\
	\implies rk\, H_1(X) + \frac{p-1}{2} \geq rk\, IH^{\bar{n}}_2(X)
\end{align*}

But the absolute value of the signature $|\sigma_{IH}(X)|$ of the intersection homology pairing on $IH^{\bar{n}}_2(X;\mathbb{Q})$ is bounded above by $rk\, IH^{\bar{n}}_2(X)$.
Lastly we invoke Lemma~\ref{calc} to see $\Xi_p(K) = \sigma_{IH}(X)$.
\end{proof}

\begin{cor}
	\label{obstruction}
	Let $K$ be a slice knot with determinant $\Delta_K(-1)\neq\pm1$ and let $p>1$ be a square-free integer dividing $\Delta_K(-1)$. Assume that for each $s$-extendible $p$-coloring of $K$ and associated $p$-fold irregular dihedral cover $M$ the inequality 
	\begin{equation}
	\label{obstrctn}
	|\Xi_p(K)|> rk\, H_1(M)+ \frac{p-1}{2}	
	\end{equation}
		 holds. Then $K$ is not homotopy ribbon.
\end{cor}

\begin{remk}
One can apply the above corollary without necessarily determining which $p$-colorings of $K$ extend over slice disks.  Let $\Sigma_2$ denote the two-fold cover branched along $K$. By Remark~\ref{metab}, every $s$-extendible $p$-coloring has the property that the induced surjection $\bar{\rho}: H_1(\Sigma_2) \twoheadrightarrow \mathbb{Z}_p$ vanishes on some metabolizer $H_D$.	
Therefore, if the inequality~(\ref{obstrctn}) holds for the (potentially larger) set of $p$-colorings that vanish on some metabolizer $H_D$, then $K$ is not homotopy ribbon.  
\end{remk}

\begin{proof}
	First, since $K$ is a knot, $|\Delta_K(-1)|$ is odd~\cite[Theorem 6.7.1]{cromwell2004knots}, so $p$ must be odd as well. Secondly, when $p$ is square-free, the knot $K$ admits a $p$-coloring if and only if $p$ divides the determinant of $K$. This follows from the discussion of existence of dihedral branched covers given in the proof of Lemma~\ref{diskext}, together with the fact that  $|\Delta_K(-1)|$ is the order of the first homology of the double branched cover of $K$~\cite[p. 149]{fox1962quick}. Hence, for any value of $p$ and coloring that extends over a slice disk, the associated value of $|\Xi_p(K)|$ can be discussed. Now apply Theorem~\ref{inequality}.
\end{proof}

The above obstruction simplifies whenever the term $rk\, H_1(M)$ vanishes. Knots with this property are discussed below.  

\subsection{Rationally admissible singularities} 		
\label{some knots}

The invariant $\Xi_p(K)$, as given in Equation~(\ref{xi}), is well-defined, provided that $K$ arises as the only singularity on the otherwise locally flat branching set of a $p$-fold dihedral cover over some four-manifold; this is a straightforward generalization of~\cite[Proposition 2.7]{kjuchukova2018dihedral} applying Theorem~\ref{sashka}. Put differently,  $\Xi_p(K)$ is an invariant of a knot $K$ together with a fixed $p$-coloring whenever this coloring extends over some locally flat surface $F$  embedded in some four-manifold $X$ with $\partial X=S^3$ and $\partial F=K$. Note also that the dihedral cover $f: W\to X$ induced by such a coloring can be extended in the obvious way to a singular dihedral branched cover $Z=W\cup c(\partial W) \to X\cup D^4$. The total space $Z$ is a manifold if and only if $f^{-1}(\partial X)\cong S^3$. More generally, $f^{-1}(\partial X)$ being a rational homology sphere is equivalent to the vanishing of $rk\, H_1(M)$, which in turn simplifies the ribbon obstruction given in Theorem~\ref{inequality}.

\begin{defn}
\label{admissible}
	Let $X$ be a four-manifold with $\partial X\cong S^3$ and let $F\subset X$ be a properly embedded locally flat surface with connected boundary $K$. If the pair $(X, F)$ admits a $p$-fold irregular dihedral branched cover $W\to X$ with $\partial W$ connected, we say that $K$ is {\it $p$-admissible over $X$}. If
	 moreover $M:=\partial W$ is a rational homology sphere, $K$ is said to be {\it rationally $p$-admissible over $X$}. When $M\cong S^3$, we call $K$  {\it strongly $p$-admissible over $X$}. 
\end{defn}

For any hr-extendible coloring of a knot $K$, when the associated dihedral cover of $K$ is a rational homology sphere, Proposition~\ref{inequality} gives $|\Xi_3(K)| \leq 1$.
Moreover, under these assumptions, $\Xi_3(K)$ is odd \cite[Theorem 5]{cahnkju2018genus}, so  the above inequality in fact implies  $\Xi_3(K) = \pm 1$.  Computing $\Xi_p$ is especially approachable for $p=3$ thanks to~\cite{cahnkjuchukova2018computing} and~\cite{cahnkjuchukova2016linking}. 
More generally, by reasoning as in the proof of Lemma~\ref{genus}, we can calculate an upper bound on $rk\, H_1(M)$ from the bridge number of $K$, which allows us to simplify (although also potentially weaken) the inequality~(\ref{ineq}). It is also possible to compute  $rk\, H_1(M)$  using Fox's method~\cite{fox1962quick} for writing down a presentation for $\pi_1(M)$ from a colored diagram of $K$.

\begin{lem}
\label{genus}
	Let $K$ be a $p$-colorable $n$-bridge knot and $f: M\to S^3$ a $p$-fold irregular dihedral cover branched along $K$. Then, the Heegaard genus of $M$, and, consequently, $rk\, H_1(M)$, is at most $\frac{1}{2}(p-1)(n-2)$.
	\end{lem}

\begin{proof}
	The argument follows the exact outlines of the proof of the well-known fact that the $p$-fold irregular dihedral cover of a two-bridge knot is the sphere. Let $K$ be a knot in $n$-bridge position and $S\subset S^3$ a bridge sphere for $K$. Write $S^3=B^3_1\cup _S B^3_2$. We know that $K\cap S$ consists of $2n$ points and $K\cap B_i^3$ is a trivial $n$-tangle for each $i=1, 2$, implying that $f^{-1}(B_i^3)$ is a handlebody of genus equal to the genus of $f^{-1}(S)$.	 Since $f$ is an irregular dihedral cover, each branch point of $f$ has $\frac{p+1}{2}$ pre-images. Therefore, we compute $\chi(f^{-1}(S))=p(2-2n)+2n\frac{p+1}{2} =
	2p-pn+n$, concluding that $g(f^{-1}(S))=\frac{1}{2}(p-1)(n-2)$. This gives the desired bound on the Heegaard genus of $M$ and, therefore, on the rank of its first homology. 
	\end{proof}

\begin{cor}\label{bridgebound}
	Let $K$ be a slice knot whose determinant satisfies $\Delta_K(-1)\neq \pm1$. If $p>1$ is a square-free integer dividing $\Delta_K(-1)$, then $|\Xi_p(K)| \leq \frac{1}{2}(p-1)(n-1)$ for any value $\Xi_p(K)$ defined using an hr-extendible $p$-coloring of $K$; also see Remark~\ref{metab}. \end{cor}
\begin{proof}
For any $p$ as above, $K$ admits $p$-colorings as in Corollary~\ref{obstruction}. Given a $p$-coloring of $K$, Lemma~\ref{genus} implies that the rank of the first homology of the induced $p$-fold irregular dihedral branched cover $M$ of $S^3$ satisfies $rk\, H_1(M; \mathbb{Z})\leq \frac{1}{2}(p-1)(n-2)$. 	
If in addition this coloring is hr-extendible, then Theorem~\ref{inequality} gives $|\Xi_p(K)| \leq \frac{1}{2}(p-1)(n-1)$ for the associated invariant $\Xi_p(K)$. \end{proof}

If $K$ is a pretzel knot, then~\cite[Theorem 1]{hosokawa19863} implies any $3$-fold irregular dihedral cover of $S^3$ branched along $K$ is either a rational homology sphere or a connect sum of a rational homology sphere with some number of copies of $S^2 \times S^1$. The number of $S^2 \times S^1$ summands, and hence the rank of the first homology of the irregular dihedral cover, can be determined via \cite{hosokawa19863}'s cancelling procedure that reduces a pretzel knot $K$ to a split sum of links, where each component link is trivially colored and is a connect sum of torus links. In particular, by applying this cancelling procedure, one can identify rationally and strongly 3-admissible pretzel knots among all knots that admit a branching surface as in Definition~\ref{admissible}. A necessary and sufficient condition for the existence of such a surface in $B^4$ is given in~\cite{kjorr2017admissible}.

Finally, we note that the results of this section can be generalized to potentially distinguish the four-genus from the homotopy ribbon four-genus for non-slice knots as well. Given a $p$-colorable knot $K$, the argument used to prove Theorem~\ref{inequality} can be applied to deduce a lower bound for the minimal genus of a homotopy ribbon surface $F\subset B^4$ with $\partial F=K$ and such that a $p$-coloring of $K$ extends over $F$. For the special case of strongly $p$-admissible knots, this was done in~\cite{cahnkju2018genus}. The lower bound obtained therein is seen to equal the homotopy ribbon genus, without regard to coloring, for a family of knots $\{K_m\}^\infty_{m=0}$ such that the homotopy ribbon genus of $K_m$ equals $m$. This genus is found by an explicit construction matching the lower bound. By Lemma~\ref{diskext}, when $m$ is positive,   the $\Xi$ invariant alone is enough to detect that the knots $K_m$ are not ribbon.\\

{\bf Acknowledgment.} The authors would like to thank Laurentiu Maxim and Kent Orr for helpful discussions. AK thanks the MPIM for its support and hospitality.\\ 		

Christian Geske\\
University of Wisconsin Madison\\
{\it cgeske@wisc.edu}\\

Alexandra Kjuchukova\\
Max Planck Institute for Mathematics-Bonn\\
{\it sashka@mpim-bonn.mpg.de}\\

Julius L. Shaneson\\
University of Pennsylvania\\
{\it shaneson@math.upenn.edu}

\nocite{ruberman2016smoothly}

\bibliographystyle{amsplain}
\bibliography{BrCovBib}

\end{document}